\documentclass[12pt]{amsart}
\usepackage{amsfonts}
\usepackage{amsthm}
\usepackage{amsmath}
\usepackage{mathtools}
\usepackage{amscd}
\usepackage[latin2]{inputenc}
\usepackage{t1enc}
\usepackage[mathscr]{eucal}
\usepackage{indentfirst}
\usepackage{graphicx}
\usepackage{graphics}
\usepackage{pict2e}
\usepackage{epic}
\numberwithin{equation}{section}
\usepackage[margin=3.5cm]{geometry}
\usepackage{epstopdf} 
\usepackage{commath}

\theoremstyle{plain}
\newtheorem{Th}{Theorem}[section]
\newtheorem{Lemma}[Th]{Lemma}

\newtheorem{Prop}[Th]{Proposition}

\theoremstyle{definition}
\newtheorem{Def}[Th]{Definition}
\newtheorem{Que}{Question}

\newtheorem{Rem}[Th]{Remark}
\newtheorem{?}[Th]{Problem}

\begin{document}

\title[On the kernel of the zero-surgery homomorphism] % 페이제 머릿글에 달리는 짧은 제목 (제목이 길 경우)
{On the kernel of the zero-surgery homomorphism from knot concordance}

\author{Dongsoo Lee}

\address{Department of Mathematics, Michigan State University, East Lansing, MI, 48824} 

\email{leedon56@msu.edu}

 \subjclass[2010]{57N70, 57M10, 57M25.}

 \keywords{Knot concordance, $\widetilde{H}$-cobordism}

\begin{abstract}
Kawauchi defined a group structure on the set of homology $S^1$$\times$$S^2$'s under an equivalence relation called $\widetilde{H}$-cobordism. This group receives a homomorphism from the knot concordance group, given by the operation of zero-surgery. It is natural to ask whether the zero-surgery homomorphism is injective. We show that this question has a negative answer in the smooth category. Indeed, using knot concordance invariants derived from knot Floer homology we show that the kernel of the zero-surgery homomorphism contains a $\mathbb{Z}^\infty$-subgroup. 
\end{abstract}

\maketitle

\section{Introduction}
In 1976, Kawauchi introduced an equivalence relation on 3-dimensional manifolds with the homology of $S^1\times S^2$ \cite{K1}. This notion, which he refers to as {\em $\widetilde{H}$-cobordism}, has the virtue of allowing a natural group structure induced by an operation $\bigcirc$ called the {\em circle union}. This group is denoted by $\Omega(S^1\times S^2)$ and is called the {\em $\widetilde{H}$-cobordism group}. An interesting feature of the $\widetilde{H}$-cobordism group is that it receives a homomorphism from the knot concordance group $\mathcal{C}$ using the zero-surgery operation. It is natural to wonder how faithfully the knot concordance group is reflected in the $\widetilde{H}$-cobordism group under this map.

\begin{Que}
Is the zero-surgery homomorphism $\omega:\mathcal{C}\to\Omega(S^1\times S^2)$ injective?
\end{Que}

Closely related to $\widetilde{H}$-cobordism is the more well-known notion of $\mathbb{Z}$-homology cobordism between 3-manifolds. A $\mathbb{Z}$-homology cobordism between $Y_0$ and $Y_1$ is a cobordism $W$ such that the inclusions $Y_i$ $\hookrightarrow$ $W$, $i=0, 1$, induce isomorphisms on integral homology groups. In \cite{CFHH}, the question of injectivity of the zero-surgery map from the knot concordance group to the set of all $\mathbb{Z}$-homology cobordism classes of 3-manifolds with the homology of $S^1\times S^2$ was addressed by Cochran, Franklin, Hedden, and Horn. 

\newtheorem*{t1}{Theorem 3.1 of \cite{CFHH}}
\begin{t1}
There exist topologically slice knots whose 0-surgeries are smoothly $\mathbb{Z}$-homology cobordant rel meridians, but which are not smoothly concordant.\
\end{t1}
That 0-surgeries on knots are smoothly $\mathbb{Z}$-homology cobordant {\em rel meridians} means that the positively-oriented meridians of knots are homologous in the first homology group of a $\mathbb{Z}$-homology cobordism between them.\

\smallskip

The following result of Yasui \cite{Y} can be used to establish that the zero-surgery homomorphism $\omega$ is not injective:

\newtheorem*{t2}{Theorem 1.9 of \cite{Y}}
\begin{t2}
There exists a pair of knots in $S^3$ with the same 0-surgery which are not smoothly concordant for any orientations. Furthermore, there exist infinitely many distinct pairs of such knots.\
\end{t2}

Perhaps the most compelling reason to study $\widetilde{H}$-cobordism is that the group structure present in this context allows one to quantify the failure of injectivity of the zero-surgery map. Despite their abundance, Yasui's examples do not lead to large subgroups of {\rm ker($\omega$)}, since the 3-manifolds obtained via 0-surgeries on the pairs have no (obvious) relationship.\

The purpose of this article is to demonstrate that {\rm ker($\omega$)} is quite large. Indeed, inspired by Cochran, Franklin, Hedden, and Horn's work, we show:

\newtheorem*{MT}{Theorem 1}
\begin{MT}
The kernel of the zero-surgery homomorphism $\omega:\mathcal{C}\to\Omega(S^1\times S^2)$ contains a subgroup isomorphic to $\mathbb{Z}^\infty$.
\end{MT}

There are a number of related questions that arise from our work. For instance, it is natural to wonder about the cokernel of the zero-surgery homomorphism.
\begin{Que}{\rm(c.f. \cite{HKMP})}
Is $\omega:\mathcal{C}\to\Omega(S^1\times S^2)$ surjective? If not, how big is the $\mathrm{coker}(\omega)$?
\end{Que}
It is also natural to ask if our result holds in the topological category.
\begin{Que}
Is $\mathrm{ker}(\mathcal{C}^{top}\to\Omega^{top}(S^1\times S^2))$ non-trivial?
\end{Que}
One might expect that in this latter category $\omega$ would be closer to an isomorphism.\
Certainly the techniques we use are manifestly smooth.\\
{\bf Outline:} In Section 2, we briefly review the $\widetilde{H}$-cobordism group $\Omega(S^1\times S^2)$ and the zero-surgery homomorphism $\omega:\mathcal{C}\to\Omega(S^1\times S^2)$. We also discuss several properties of the knot concordance invariants $\Upsilon$, $\tau$ and $\{V_i\}$ derived from knot Floer homology. In Section 3, we establish a relationship between satellite operations and $\widetilde{H}$-cobordism. Using the aforementioned knot invariants, in Section 4 we show that there is a $\mathbb{Z}^\infty$-subgroup in $\mathrm{ker}(\omega)$.

\section*{Acknowledgement}
I would like to thank my advisor Matthew Hedden for his willingness and patience to teach me the subject and his support. I also thank Yewon Joung for her help and conversations.

\section{Preliminaries}

\subsection{An overview of the $\widetilde{H}$-cobordism group $\Omega(S^1\times S^2)$}
In this subsection, we review the definitions and basic properties of Kawauchi's $\widetilde{H}$-cobordism group. We refer the reader to \cite{K1} for more details.\

A 3-dimensional {\em homology orientable handle} is a compact, orientable 3-manifold whose integral homology groups are isomorphic to those of $S^1\times S^2$. A {\em distinguished homology handle} is a pair $(Y, \alpha)$ consisting of an oriented homology handle $Y$ and a specified generator $\alpha$ of $H_1(Y;\mathbb{Z})$.\

\begin{Def}
Two distinguished homology handles $(Y_0,\alpha_0)$ and $(Y_1,\alpha_1)$ are {\em $\widetilde{H}$-cobordant} if there is a compact, connected, and oriented 4-dimensional manifold $W$ with $\partial W=-Y_0\sqcup Y_1$ and a cohomology class $\varphi\in H^1(W;\mathbb{Z})$ such that
\begin{enumerate}
\item $\varphi_{|_{Y_i}}$ are dual to $\alpha_i$ for $i=0, 1$,
\item $H_*(\widetilde{W}_\varphi;\mathbb{Q})$ is finitely generated over $\mathbb{Q}$ for each $*$, where $\widetilde{W}_\varphi$ is the infinite cyclic covering of $W$ associated with $\varphi$.
\end{enumerate}
If they are $\widetilde{H}$-cobordant, we write $(Y_0,\alpha_0)\sim(Y_1,\alpha_1)$ and call $(W,\varphi)$ (or simply $W$)  an {\em $\widetilde{H}$-cobordism} between $(Y_0,\alpha_0)$ and $(Y_1,\alpha_1)$ (or between $Y_0$ and $Y_1$). 
\end{Def}

\begin{Lemma}\cite{K1}\label{lem1}
$\widetilde{H}$-cobordism is an equivalence relation.
\end{Lemma}
\begin{proof}
The symmetry of the relation is trivial and the transitivity can be checked using the Mayer-Vietoris sequence. We verify reflexivity by showing that $H_i(\widetilde{Y}; \mathbb{Q})$ is finitely generated, where $Y$ is an oriented homology handle and $\widetilde{Y}$ is the covering space associated with a cohomology class in $H^1(Y;\mathbb{Z})$ dual to a generator of $H_1(Y;\mathbb{Z})$. In \cite[Proof of Assertion 5]{M1}, it is shown that if $H_1(Y; \mathbb{Q})\cong\mathbb{Q}$, then $H_1(\widetilde{Y}; \mathbb{Q})$ is finitely generated by using the Milnor exact sequence for the cover $\widetilde{Y}\to Y$. By the partial Poincar\'e duality theorem, see \cite[Theorem 2.3]{K3}, $H^0(\widetilde{Y}; \mathbb{Q})\cong H_2(\widetilde{Y}; \mathbb{Q})$ since $H_i(\widetilde{Y};\mathbb{Q})$ is finitely generated for $i=0, 1$. So, $H_2(\widetilde{Y}; \mathbb{Q})\cong \mathbb{Q}$.
\end{proof}

\begin{Lemma}\label{lem2}
If there is an orientation-preserving diffeomorphism $f:(Y_0,\alpha_0)\to(Y_1,\alpha_1)$ with $f_*(\alpha_0)=\alpha_1$, then $(Y_0,\alpha_0)\sim(Y_1,\alpha_1)$.
\end{Lemma}
\begin{proof}
Let $W_0=Y_0\times[0,1]$ and $W_1=Y_1\times[0,1]$. In the proof of Lemma \ref{lem1}, we checked that $W_0$ and $W_1$ are $\widetilde{H}$-cobordisms. Let $N_0$ and $N_1$ be the collar neighborhoods of $Y_0\times1$ and $Y_1\times0$, respectively. Then $N_0\simeq Y_0\times(1-\epsilon,1]$ and $N_1\simeq Y_1\times[0,\epsilon)$. Define $$W=\frac{(W_0 \setminus (Y_0\times 1))\sqcup (W_1 \setminus (Y_0\times 0))}{(x,1-\theta) \sim (f(x),\theta)}$$ for $0<\theta<\epsilon$. It is clear that $W$ is a smooth 4-manifold with $\partial W=-Y_0\sqcup Y_1$. Moreover, the infinite cyclic covering $\widetilde{W}$ of $W$ associated with the dual of $\alpha_0$ (or $\alpha_1$) is the union of $\widetilde{W}_{0, \alpha_{0}^{*}}$ and $\widetilde{W}_{1, \alpha_{1}^{*}}$, where their intersection is $\widetilde{Y}_0\times (1-\epsilon,1)$ (or $\widetilde{Y}_1\times (0,\epsilon)$). From the Mayer-Vietoris sequence, the homology groups of $\widetilde{W}$ over $\mathbb{Q}$ are finitely generated since those of $\widetilde{W}_{0, \alpha_{0}^{*}}$ and $\widetilde{W}_{1, \alpha_{1}^{*}}$ are finitely generated, so $W$ is an $\widetilde{H}$-cobordism between $(Y_0,\alpha_0)$ and $(Y_1,\alpha_1)$.
\end{proof}

Using the above Lemma \ref{lem2}, we see that $(S^1\times S^2,\alpha_{\scriptscriptstyle S^1\times S^2})$, $(-(S^1\times S^2),\alpha_{\scriptscriptstyle S^1\times S^2})$, $(S^1\times S^2,-\alpha_{\scriptscriptstyle S^1\times S^2})$, and $(-(S^1\times S^2),-\alpha_{\scriptscriptstyle S^1\times S^2})$ are all $\widetilde{H}$-cobordant, where $\alpha_{\scriptscriptstyle S^1\times S^2}$ is the homology class of $S^1\times*$ with a fixed orientation. Indeed, there are obvious orientation-preserving diffeomorphisms between them.\

If a distinguished homology handle $(Y,\alpha)$ is $\widetilde{H}$-cobordant to $(S^1\times S^2, \alpha_{\scriptscriptstyle S^1\times S^2})$, then it is called {\em null $\widetilde{H}$-cobordant}.\

It can be easily checked that $(Y,\alpha)$ is null $\widetilde{H}$-cobordant if and only if there is a compact, connected, and oriented 4-manifold $W^+$ with $\partial W^+=Y$, and a class $\varphi\in H^1(W^+;\mathbb{Z})$ such that $\varphi_{|_{Y}}=\alpha^*$ and $H_*(\widetilde{W}^+_\varphi;\mathbb{Q})$ is finitely generated. In this case, $(W^+, \varphi)$ (or $W^+$) is called a {\em null $\widetilde{H}$-cobordism} of $(Y,\alpha)$ (or Y).

\begin{Def}
$\Omega(S^1\times S^2)$ is defined to be the set of all distinguished homology handles modulo the $\widetilde{H}$-cobordism relation. We will denote elements of $\Omega(S^1\times S^2)$ by $[(Y,\alpha)]$ and $[(S^1\times S^2, \alpha_{\scriptscriptstyle S^1\times S^2})]$ by $0$.
\end{Def}

Now, we introduce a group operation on $\Omega(S^1\times S^2)$. This operation is defined by round 1-handle attachment along curves representing the specified generators of $H_1$.\

In more detail, let $(Y_0,\alpha_0)$ and $(Y_1,\alpha_1)$ be distinguished homology handles. For each $i=0, 1$, choose a smoothly embedded simple closed oriented curve $\gamma_i$ in $Y_i$ such that $[\gamma_i]=\alpha_i$ in $H_1(Y_i; \mathbb{Z})$. Then there exists a closed connected orientable surface $F_i$ in $Y_i$ which intersects $\gamma_i$ in a single point. Let $\nu(\gamma_i)$ be a tubular neighborhood of $\gamma_i$. Then $\nu(\gamma_i)$ is diffeomorphic to $S^1\times B^2$. Choose smooth embeddings $$h_0:S^1\times B^2\times0\to Y_0,$$ $$h_1:S^1\times B^2\times1\to Y_1$$ for $\nu(\gamma_i)$ such that
\begin{enumerate}
\item there exist points $s\in S^1$ and $b\in\mathrm{Int}(B^2)$ with $h_i(s\times B^2\times i)\subset F_i$ and $h_i(S^1\times b\times i)=\gamma_i$,
\item $h_i$ is orientation reversing with respect to the orientation of $S^1\times B^2\times i$ induced from an orientation of $S^1\times B^2\times [0,1]$.
\end{enumerate}
Let $Y_{i}(\gamma_i)=Y_i\setminus\mathrm{Int}(\nu({\gamma_i}))$ and $\bar{h_{i}}=h_{i}|_{S^1\times\partial B^2\times i}$. Now, define $$Y_0\bigcirc Y_1\coloneqq Y_{0}(\gamma_0)\cup_{\bar{h_{0}}}(S^1\times\partial B^2\times [0,1])\cup_{\bar{h_{1}}}Y_{1}(\gamma_1).$$
We claim $Y_0\bigcirc Y_1$ is an oriented homology handle. To see this, let $\bar{b}\in\partial B^2$ and $\bar{\gamma_{i}}= h_i(S^1\times \bar{b}\times i)\subset Y_i$. We give an orientation to $\bar{\gamma_{i}}$ so that $[\bar{\gamma_{i}}]=[\gamma_i]$ in $H_1(Y_i;\mathbb{Z})$. Let $\mu_i=h_i(s\times \partial B^2 \times i)\subset Y_i$. We can check that $[\bar{\gamma_{i}}]$ is a generator of $H_1(Y_{i}(\gamma_i);\mathbb{Z})\cong\mathbb{Z}$ and $[\mu_i]=0$ in $H_1(Y_{i}(\gamma_i);\mathbb{Z})$ since $\mu_i$ bounds an orientable surface $F_i\setminus h_i(s\times \mathrm{Int}(B^2)\times i)$ in $Y_{i}(\gamma_i)$. From the Mayer-Vietoris sequence, we conclude that $H_1(Y_0\bigcirc Y_1;\mathbb{Z})\cong\mathbb{Z}$. Since $Y_0\bigcirc Y_1$ is orientable, its homology groups are isomorphic to those of $S^1\times S^2$ by Poincar\'e duality.\

From the above construction, we give an orientation to $Y_0\bigcirc Y_1$ induced by the orientation of $Y_0$ and $Y_1$ and the generator $\alpha$ of $H_1(Y_0\bigcirc Y_1 ; \mathbb{Z})$ can be specified by the homology class of $\bar{\gamma_{0}}$ or $\bar{\gamma_{1}}$, which are homologous in $Y_0\bigcirc Y_1$. 

\begin{Def}
For two distinguished homology handles $(Y_0,\alpha_0)$ and $(Y_1,\alpha_1)$, we define $(Y_0,\alpha_0)\bigcirc (Y_1,\alpha_1)$ to be the distinguished homology handle $(Y_0\bigcirc Y_1,\alpha)$ constructed as above and call it a {\em circle union} of $(Y_0,\alpha_0)$ and $(Y_1,\alpha_1)$.
\end{Def}

The circle union operation satisfies the following properties.

\begin{Prop}\cite{K1}\label{prop1}
\begin{enumerate}
\item $(Y_0,\alpha_0)\bigcirc (Y_1,\alpha_1)\sim(Y_0,\alpha_0)\bigcirc^{'}(Y_1,\alpha_1)$, where $\bigcirc, \bigcirc^{'}$ are circle unions with different choices of $\gamma$'s and $h$'s used above, i.e., $[(Y_0,\alpha_0)\bigcirc (Y_1,\alpha_1)]$ is well-defined.
\item $(Y_0,\alpha_0)\sim(Y_1,\alpha_1)$ if and only if $(Y_0,\alpha_0)\bigcirc (-Y_1,\alpha_1)$ is null $\widetilde{H}$-cobordant. 
\item If $(Y_0,\alpha_0)$ and $(Y_1,\alpha_1)$ are null $\widetilde{H}$-cobordant, then $(Y_0,\alpha_0)\bigcirc (Y_1,\alpha_1)$ is null $\widetilde{H}$-cobordant.
\end{enumerate}
\end{Prop}
We sketch the proof of (1). For more details see \cite{K1}.
\begin{proof}[sketch of the proof of (1)]\
Given distinguished homology handles $(Y_0,\alpha_0)$ and $(Y_1, \alpha_1)$, we start with product cobordisms $Y_0\times[0,1]$ and $Y_1\times[0,1]$. At the level $t=1$, we attach a 4-dimensional round 1-handle $S^1\times B^2\times [0,1]$ along $\gamma$'s using $h$'s we chose above. At the level $t=0$, however, we choose different curves and embeddings from $\gamma$'s and $h$'s, and attach another round 1-handle $S^1\times B^2\times [0,1]$ along them. Let $W$ denote the 4-dimensional manifold obtained by attaching two round 1-handles to $Y_0\times [0,1] \sqcup Y_1\times [0,1]$. Then it is clear that $\partial W = Y_0\bigcirc Y_1 \sqcup -(Y_0\bigcirc^{'} Y_1)$, where $Y_0\bigcirc^{'} Y_1$ is a circle union of $Y_0$ and $Y_1$ using the different curves and embeddings chosen at level $t=0$. See Figure \ref{well} for a schematic picture of $W$. 
\begin{figure}
 \centering
 \includegraphics[scale=0.13]{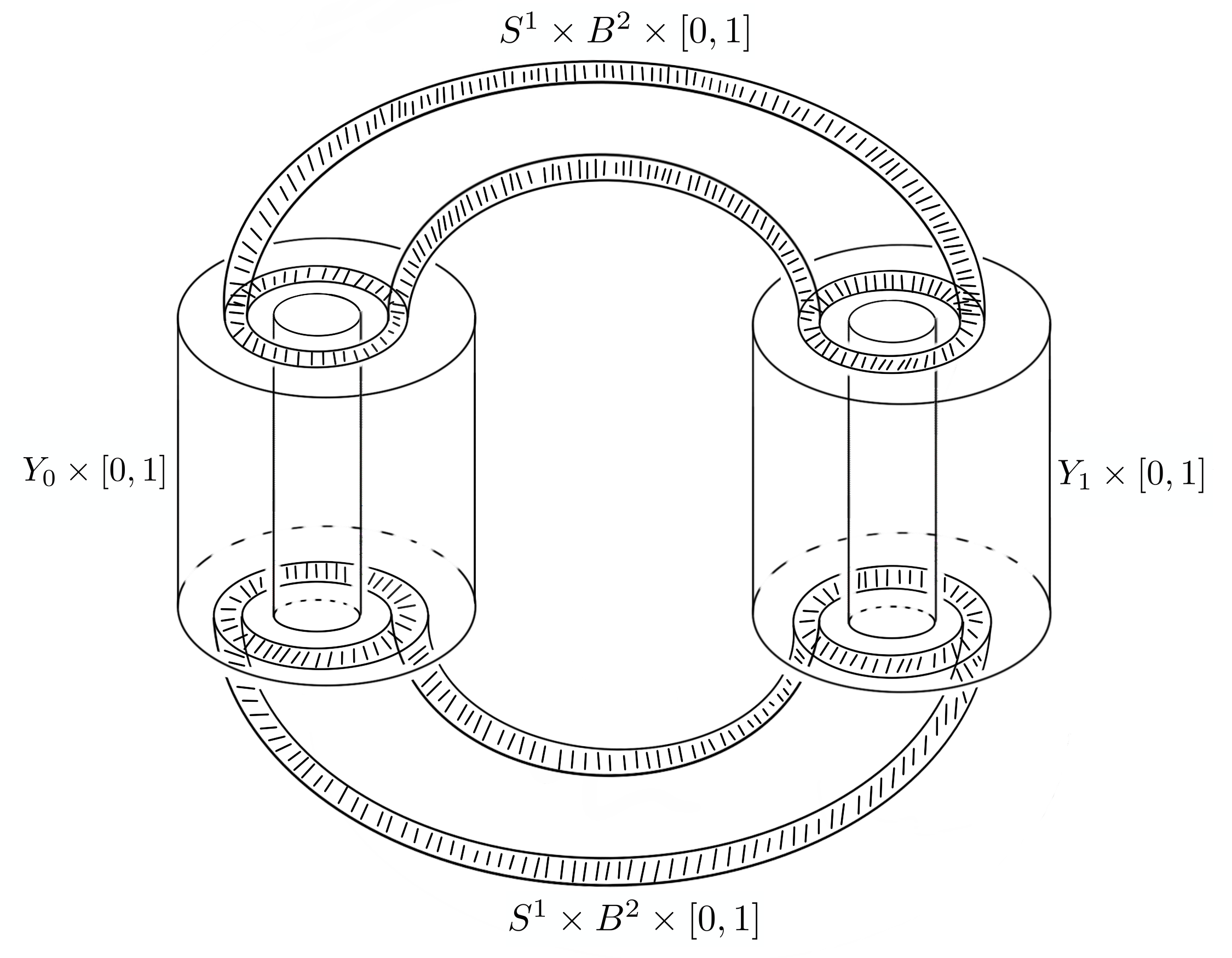}
 \caption{}
 \label{well}
\end{figure}
By using the Mayer-Vietoris sequence, we can easily check that $W$ is a $\widetilde{H}$-cobordism between $Y_0\bigcirc Y_1$ and $Y_0\bigcirc^{'} Y_1$.
\end{proof}

\begin{Rem}
The circle union of homology handles depends on the choices of curves representing specified generators and embeddings, and different choices can yield different homology handles, up to diffeomorphisms. The circle union operation, however, is well-defined under $\widetilde{H}$-cobordism by Proposition \ref{prop1}(1). One might wonder whether this operation is well-defined under $\mathbb{Z}$-homology cobordism. Since the $\widetilde{H}$-cobordism $W$ constructed in the proof of Proposition \ref{prop1}(1) has one more $\mathbb{Z}$-summand in $H_1(W;\mathbb{Z})$, it is not a $\mathbb{Z}$-homology cobordism.
\end{Rem}

Proposition \ref{prop1} leads to the following theorem.
\begin{Th}\cite[Theorem 1.9]{K1}
The set  $\Omega(S^1\times S^2)$ is an abelian group under the sum $[(Y_0,\alpha_0)]+[(Y_1,\alpha_1)]=[(Y_0,\alpha_0)\bigcirc (Y_1,\alpha_1)]$, with identity $0=[(S^1\times S^2,\alpha_{\scriptscriptstyle S^1\times S^2})]$. The inverse $-[(Y,\alpha)]$ of $[(Y,\alpha)]$ is $[(-Y,\alpha)]$.
\end{Th}

Next, we define the zero-surgery homomorphism $\omega$ from the knot concordance group $\mathcal{C}$ to $\Omega(S^1\times S^2)$.\

For any oriented knot $K\subset S^3$, let $S_{0}^{3}(K)$ be the closed 3-manifold obtained from 0-surgery along a knot $K$. It is easily checked that $S_{0}^{3}(K)$ is an oriented homology handle, i.e., the homology groups of $S_{0}^{3}(K)$ are isomorphic to those of $S^1\times S^2$. We give an orientation to the meridian $m$ so that the linking number with $K$ is +1. Then the homology class $[m]$ represents a generator of $H_1(S_{0}^{3}(K);\mathbb{Z})$. We define $\omega(K)$ to be the distinguished homology handle $(S_{0}^{3}(K),[m])$. Sometimes, we write $S_0^3(K)$ for $\omega(K)=(S_0^3(K),[m])$ as the generator $[m]$ is well-understood.

\begin{Lemma}\cite[Lemma 2.4]{K1}
The map $\omega$ from the set of knots to the set of distinguished homology handles induces a homomorphism from the knot concordance group $\mathcal{C}$ to $\Omega(S^1\times S^2)$, i.e., $$(S_0^3(K_1\#K_2), [m_0]=[m_1]) \sim (S_0^3(K_1),[m_0])\bigcirc (S_0^3(K_2),[m_1]),$$ where $m_i$ is the meridian of a knot $K_i$ for each $i=0, 1$. 
\end{Lemma}
\begin{proof}
Let $K_1$ and $K_2$ be knots in $S^3$. Then the exterior $X(K_1\#K_2)$ of the connected sum of $K_1$ and $K_2$ is the quotient space of the exteriors $X(K_1)$ and $X(K_2)$ of $K_1$ and $K_2$, respectively, formed by identifying annular neighborhoods of their meridians. So, $S_0^3(K_1\#K_2)=S_0^3(K_1)\bigcirc S_0^3(K_2)$. Hence, it is sufficient to show that if $K$ is a slice knot, then $(S_0^3(K), [m])$ is null $\widetilde{H}$-cobordant. Let $B^4$ be a 4-ball with $K$ in $S^3=\partial B^4$. Since $K$ is slice, there is a smoothly embedded disk $D^2$ in $B^4$ such that $\partial D^2=K\subset\partial B^4$. Let $W=B^4\setminus\mathrm{Int}(\nu(D^2))$, where $\nu(D^2)$ is a closed tubular neighborhood of $D^2$ in $B^4$. Then $W$ has the homology of a circle by Alexander duality. Moreover, $\partial W$ is $S_0^3(K)$. Since the map $i_*:H_1(\partial W;\mathbb{Z})\to H_1(W;\mathbb{Z})$ induced by inclusion is an isomorphism, we can choose a generator $i_*([m])$ of $H_1(W;\mathbb{Z})$, where $m$ is a meridian of $K$ with linking number +1 with $K$. By \cite[Assertion 5]{M1}, the infinite cyclic covering $\widetilde{W}$ of $W$ associated with the dual of $i_*([m])$ has finitely generated homology over $\mathbb{Q}$ since $W$ has the homology of $S^1$. Thus, $(S_0^3(K),[m])$ is null $\widetilde{H}$-cobordant.
\end{proof}

\subsection{Knot concordance invariants from knot Floer homology}\label{KCI}
We now briefly discuss the knot concordance invariants $\Upsilon$, $\tau$ and $\{V_i | i \in \mathbb{Z}\}$ without giving the definitions in detail. These are all derived from knot Floer homology. For introductions and details, see \cite{BCG}, \cite{C}, \cite{HW}, \cite{H}, \cite{L}, \cite{Li}, \cite{CM}, \cite{OSS}, \cite{OS} and \cite{R}.\\

In \cite{OS}, Ozsv\'ath and  Szab\'o defined the {\em tau} invariant $\tau$, which is a group homomorphism from $\mathcal{C}$ to $\mathbb{Z}$, i.e., $\tau(K_{1} \# K_2) = \tau(K_1) + \tau(K_2)$ and $\tau(K)=0$ for any slice knot $K$.

\begin{Th}\cite{L}\label{L1}
Let $P$ be the Mazur pattern shown in Figure \ref{PC}. If $\tau(K)>0$, then $\tau(P(K))=\tau(K)+1$.
\end{Th}
\smallskip
\begin{Rem}
Theorem \ref{L1} shows that for {\em any} knot $K$ with $\tau(K)>0$, $K$ is not concordant to $P(K)$. If one merely wants to find examples of knots for which $\tau(P(K))=\tau(K)+1$, one can appeal to the slice-Bennequin inequality satisfied by $\tau$ \cite{P}. For details, see \cite[Theorem 3.1 and Corollary 3.2]{CFHH}.
\end{Rem}
\smallskip
The definition of a satellite operation $P(K)$ will be given in Section \ref{s}.\\

In \cite{OSS}, Ozsv\'ath, Stipsicz and Szab\'o introduced the {\em Upsilon} invariant $\Upsilon$. This is a homomorphism $\Upsilon:\mathcal{C}\to PL([0,2],\mathbb{R})$, $K\mapsto\Upsilon_{K}(t)$, $\Upsilon_K:[0,2]\to\mathbb{R}$, where $PL([0,2],\mathbb{R})$ is the group of piecewise-linear functions on $[0,2]$.

\begin{Th}\cite{OSS}\label{OS}
The invariants $\Upsilon_K(t)$ bound the slice genus of $K$. Indeed, for $0\le t\le1$,  $\abs{\Upsilon_K(t)}\le tg_s(K)$.
\end{Th}
\smallskip
\begin{Th}\cite{OSS}\label{OSS} {\rm(c.f. \cite{Li})}
 The invariant $\Upsilon$ has the following properties:
 \begin{enumerate}
 \item $\Upsilon_K(2-t) = \Upsilon_K(t)$.
 \item $\Upsilon_K(0) = 0$.
 \item $\Upsilon_{K}'(0) = -\tau(K)$.
 \item $\Upsilon_{K_1\#K_2}(t)=\Upsilon_{K_1}(t)+\Upsilon_{K_2}(t)$.
 \item $\Upsilon_{-K}(t)=-\Upsilon_K(t)$.
 \item There are only finitely many singularities of $\Upsilon_K(t)$.
 \item The derivative of $\Upsilon_K(t)$, where it exists, is an integer.
 \end{enumerate}
\end{Th}
\smallskip
Note that Theorem \ref{OS} and Theorem \ref{OSS}(4) imply $\Upsilon:\mathcal{C}\to PL([0,2],\mathbb{R})$ is a homomorphism.\\

In \cite{R}, Rasmussen introduced the {\em local h-invariants}, denoted $\{V_i|i\in\mathbb{Z}\}$ in \cite{NW}, which are a family of integer-valued knot concordance invariants.

\begin{Th}\cite{BCG}\label{BCG}
Let $K_1$, $K_2$ be two knots in $S^3$. Then for any non-negative integers $i_1$, $i_2$,$$V_{i_1+i_2}(K_1\#K_2)\le V_{i_1}(K_1)+V_{i_2}(K_2).$$
\end{Th}

\bigskip
The following results from Chen's thesis \cite{C} will be very useful.
\begin{Th}\cite{C}\label{C1}
For any knot $K$, $-2V_0(K)\le\Upsilon_K(t)\le 2V_0(-K)$.
\end{Th}
\smallskip
\begin{Prop}\cite{C}\label{C2}
Let $\{K_n|n\in\mathbb{Z}^+\}$ be a family of knots such that $$\lim_{n\to\infty}\frac{\tau(K_n)}{V_0(K_n)}=\infty,$$ then there exists a subset of $\{K_n|n\in\mathbb{Z}^+\}$ which generates a $\mathbb{Z}^\infty$-subgroup in $\mathcal{C}$.
\end{Prop}

\section{$\tilde{H}$-cobordism and satellite knots}\label{s}
Let $K$ be a knot in $S^3$. Let $P$ be a knot in a solid torus $S^1\times D^2$. Let $p:S^1\times D^2\to S^3$ be an embedding which identifies a regular neighborhood of a knot $K$ with $S^1\times D^2$ so that $p(S^1\times pt)$ is the Seifert framing of $K$. Then the knot $P(K)$ is defined to be the image of $P$ in $S^1\times D^2$ under the map $p$. $P^n(K)$ is defined to be $P(P^{n-1}(K))$ and $P^0(K)=K$. $P(K)$ is called a {\em satellite knot} with {\em pattern} $P$ and {\em companion} $K$. See Figure \ref{PC}. The {\em winding number} of $P$ is the algebraic intersection number of $P$ with a meridian disk of the solid torus.
\begin{figure}
 \centering
 \includegraphics[scale=0.15]{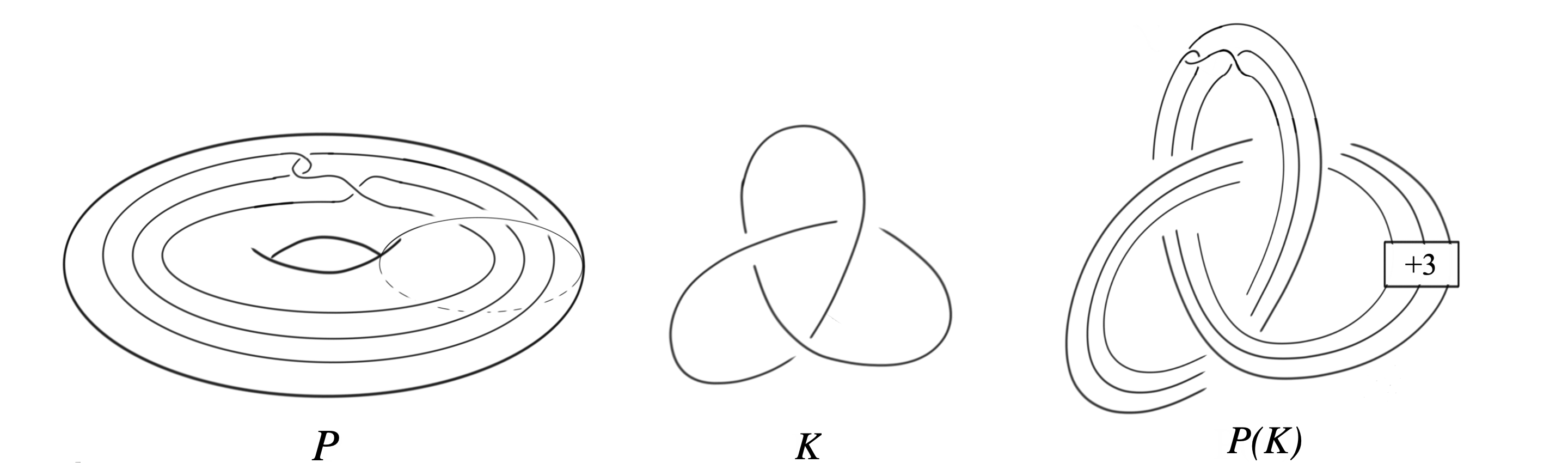}
 \caption{Satellite operation}
 \label{PC}
\end{figure}

\begin{Th}\label{thm1}
Suppose that $P$ is a pattern knot with winding number $\pm1$, and that $P(U)$ is a trivial knot in $S^3$, where $U$ is an unknot. Then for any knot $K$, $S_{0}^{3}(K)\sim S_{0}^{3}(P(K))$.
\end{Th}
\begin{proof}
We will construct an $\widetilde{H}$-cobordism $W$ between $S_{0}^{3}(K)$ and $S_{0}^{3}(P(K))$ to show that $S_{0}^{3}(K)\sim S_{0}^{3}(P(K))$. Let $X^1$ be the 4-manifold by attaching a 1-handle to the outgoing boundary of the 4-manifold $X=S_{0}^{3}(K)\times[0,1]$. This boundary is depicted in Figure \ref{PK}(a), where we replace the dotted circle $P(U)$ typically used to denote a 1-handle with a zero-framed curve since the resulting boundaries are diffeomorphic. Now let $W$ be the 4-manifold obtained by attaching a 0-framed 2-handle to $\partial^+X^1$ along the red circle shown Figure \ref{PK}(b). Because $P(U)$ is an unknot, using an isotopy from $P(U)$ to a trivial unknot, we have the following Figure \ref{PK}(c). See Figure \ref{XW} for schematic pictures of $X$, $X^1$ and $W$. By handle slides, one can show that $\partial^+W\simeq S_{0}^{3}(P(K))$, see \cite[Theorem 2.1]{CFHH}.

\begin{figure}
\centering
\includegraphics[scale=0.13]{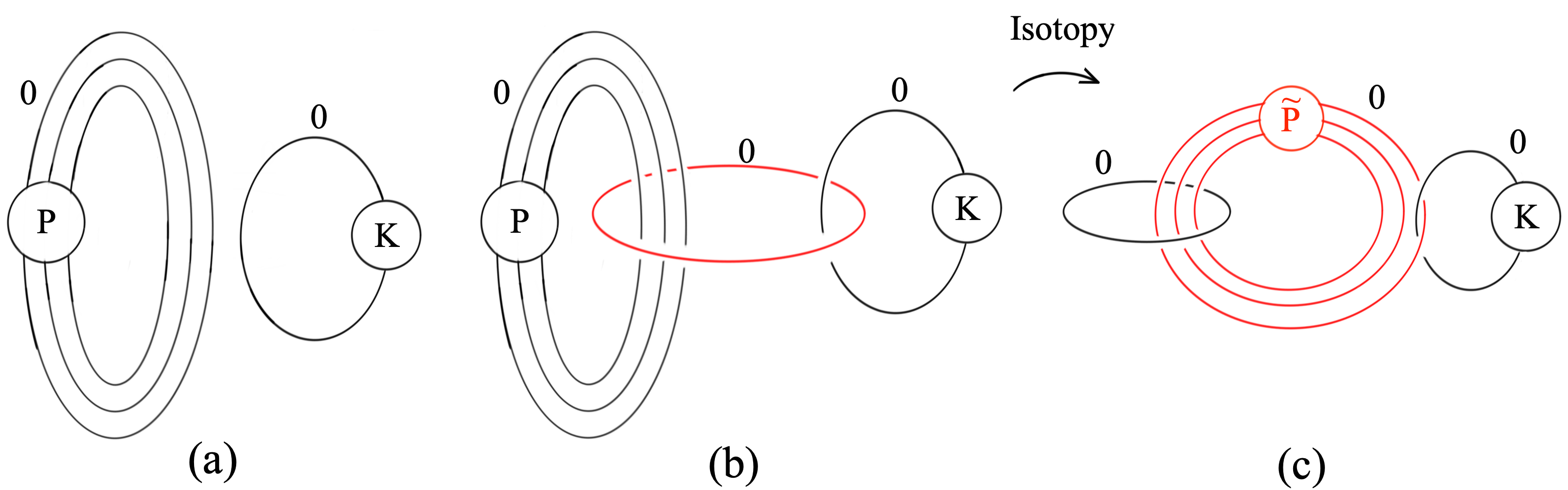}
\caption{(a) : $\partial^+X^1$ \ (b),(c) : $\partial^+W$}
\label{PK}
\end{figure}

\begin{figure}
\centering
\includegraphics[scale=0.1]{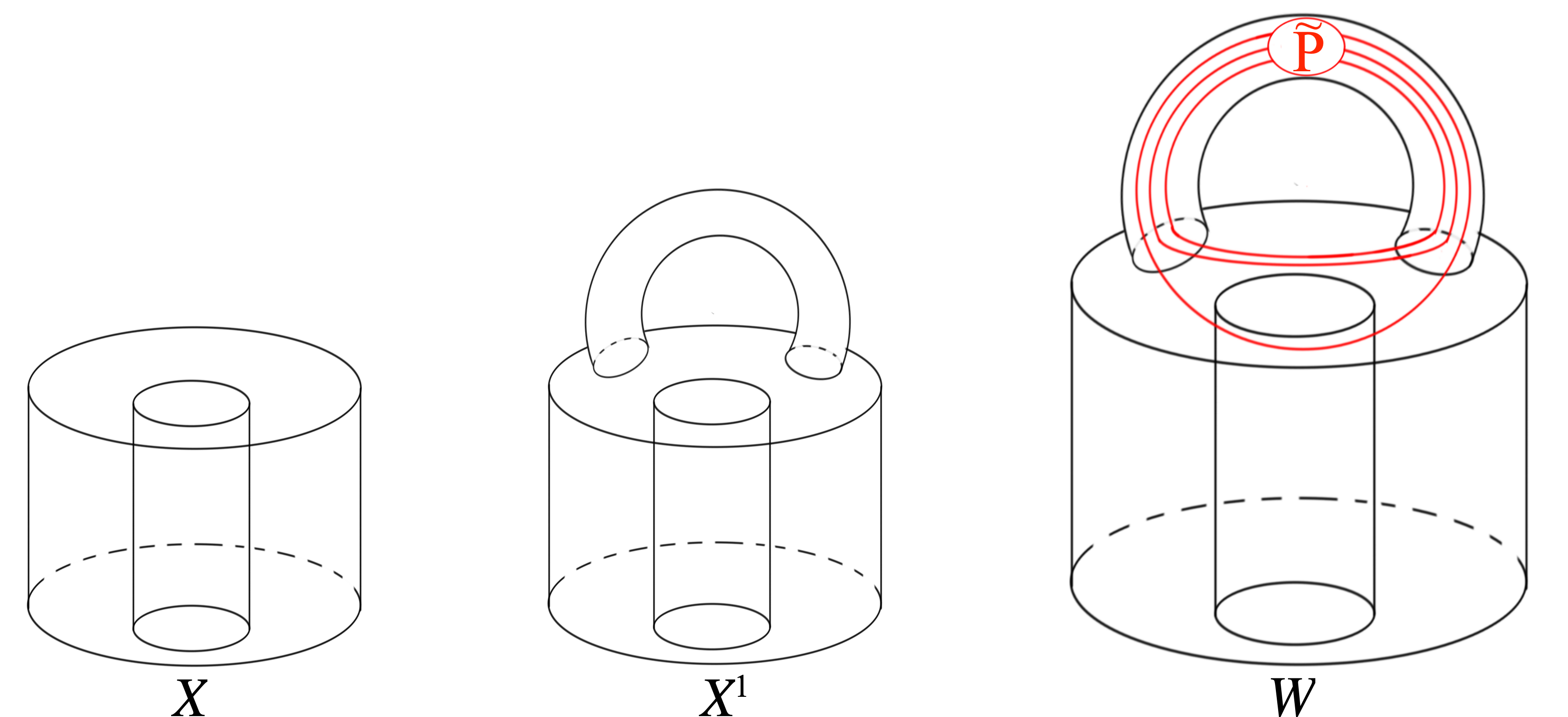}
\caption{$X, X^1$ and $W$}
\label{XW}
\end{figure}

We now show that this cobordism is an $\widetilde{H}$-cobordism, i.e., there is a cohomology class $\varphi$ for which the infinite cyclic covering $\widetilde{W}_\varphi$ of $W$ has finitely generated rational homology groups. Let $\pi_1(X)=\langle x_1, x_2,\dots,x_k|r_1, r_2,\dots,r_l,\lambda\rangle$ be the fundamental group of $X$, where $\langle x_1, x_2,\dots,x_k|r_1, r_2,\dots,r_l\rangle$ is the knot group of $K$ and $\lambda$ is a relator coming from the disk bounded by a 0-framed longitude. The abelianization map $A_X$ from $\pi_1(X)$ to $H_1(X;\mathbb{Z})$ provides a generator $[x_1]=[x_2]=\dots=[x_k]$ of $H_1(X;\mathbb{Z})\cong\mathbb{Z}$, which corresponds to the meridian of the knot $K$. Corresponding to $\mathrm{ker}(A_X)$, we have the infinite cyclic covering $\widetilde{X}$ of $X$ associated with $A_X$, and $\pi_1(\widetilde{X})\cong\mathrm{ker}(A_X)$. Indeed, $\mathrm{ker}(A_X)$ is the commutator subgroup $[\pi_1(X),\pi_1(X)]$ of $\pi_1(X)$ and the covering $\widetilde{X}$ is the {\em universal abelian covering space}. Attaching the 1-handle to the boundary $\partial^+X$ adds one extra generator $b$ to the presentation of $\pi_1$, so $\pi_1(X^1)=\langle x_1, x_2,\dots,x_k, b|r_1, r_2,\dots,r_l, \lambda\rangle$, where we orient $b$ to be compatible with winding number of the pattern $P$. Let $A_{X^1}$ be the abelianization map from $\pi_1(X^1)$ to $H_1(X^1; \mathbb{Z})\cong\mathbb{Z}\langle[x_1]=[x_2]=\dots=[x_k]\rangle\oplus\mathbb{Z}\langle[b]\rangle$. When we attach the 0-framed 2-handle to $\partial^+X^1$ to get $W$, the attaching region is homologous to $[x_1]+[b]$. Thus, the homomorphism $\varphi_1:H_1(X^1; \mathbb{Z})\to\mathbb{Z}$ defined by $[x_i]+[b]\mapsto0$ and $[x_i]\mapsto1$ can be considered as a map from $H_1(X^1; \mathbb{Z})$ to $H_1(W; \mathbb{Z})\cong\mathbb{Z}\langle[x_1]=[x_2]=\dots=[x_k]=-[b]\rangle$. Let $$\psi_1\coloneqq \varphi_1\circ A_{X^1}: \pi_1(X^1) \to H_1(X^1;\mathbb{Z}) \to H_1(W;\mathbb{Z})\cong \mathbb{Z}.$$ Note that the attaching region of the 1-handle in $\partial^{+}X$ is a disjoint union of two 3-balls, which are simply-connected. So, the attaching region can be lifted to $\widetilde{X}$. Thus, the infinite cyclic covering $\widetilde{X}^1$ of $X^1$ associated with $\varphi_1$ is obtained by attaching infinitely many 1-handles to the infinite cyclic covering $\widetilde{X}$ of $X$. See Figure \ref{CS}. 

\begin{figure}
\centering
\includegraphics[scale=0.17]{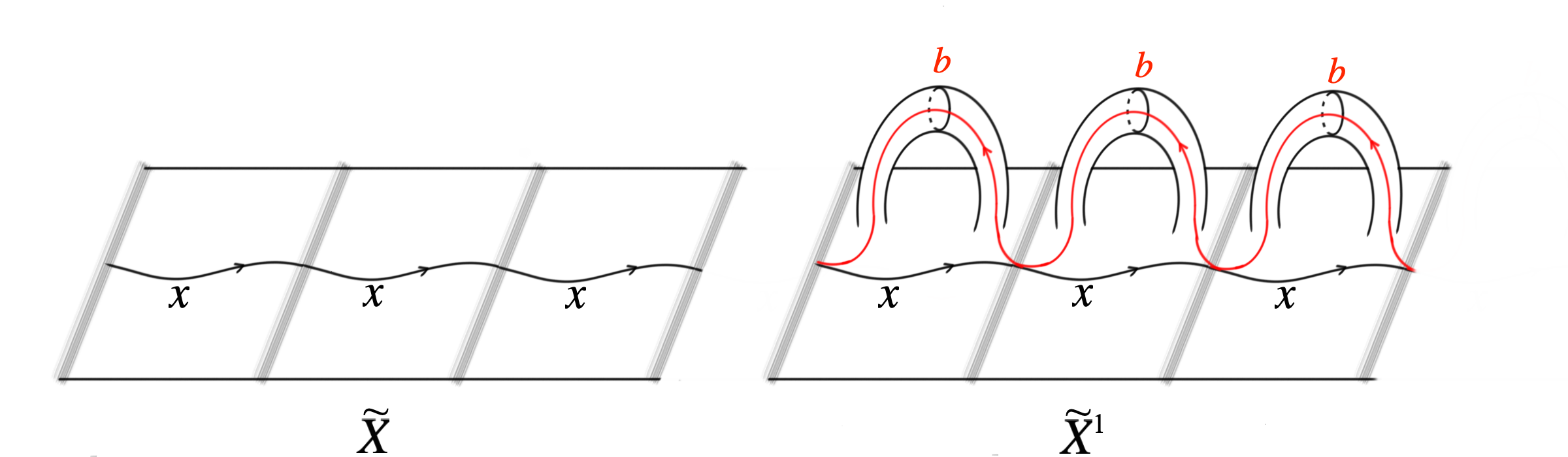}
\caption{$\widetilde{X}$ and $\widetilde{X}^1$}
\label{CS}
\end{figure}

It follows that$$
H_i(\widetilde{X}^1; \mathbb{Q})=
\begin{cases}
H_i(\widetilde{X}; \mathbb{Q}) &, i \ne 1\\
H_i(\widetilde{X}; \mathbb{Q})\oplus \mathbb{Q}[t,t^{-1}] &, i=1,
\end{cases}
$$where $t$ is a generator of the deck transformation group of the covering spaces. 

The attaching region of the 2-handle in $\partial^+X^1$ is homotopic to $x_{1}b$ and is contained in $\mathrm{ker}(\psi_1)$, which is the image of $\pi_1(\widetilde{X}^1)$ under the covering map. Hence, the attaching region of the 2-handle can be lifted to the covering $\widetilde{X}^1$. Attaching the 2-handle to $\partial^+X^1$ adds a relator $x_1b$ to the presentation of $\pi_1$, so we have
\begin{eqnarray}
\pi_1(W) & = & \langle x_1, x_2,\dots,x_k, b|r_1, r_2,\dots,r_l, x_1b\rangle\nonumber\\
 & = & \langle x_1, x_2,\dots,x_k, x_1^{-1}|r_1, r_2,\dots,r_l\rangle\nonumber\\
 & = & \langle x_1, x_2,\dots,x_k|r_1, r_2,\dots,r_l\rangle \cong \pi_1(X),\nonumber
\end{eqnarray}
and
\begin{eqnarray}
H_1(W) & = & \frac{\mathbb{Z}\langle[x_1]=[x_2]=\dots=[x_k]\rangle\oplus \mathbb{Z}\langle[b]\rangle}{<[x_1]+[b]>}\nonumber\\
& = & \mathbb{Z}\langle[x_1]=[x_2]=\dots=[x_k]=[-b]\rangle \cong \mathbb{Z}. \nonumber
\end{eqnarray}

Let $\varphi$ be the dual cohomology class of $[x_1]=[x_2]=\dots=[x_k]=[-b]$ in $H_1(W; \mathbb{Z})$. It is clear that $\varphi| _{\partial^{\pm}W}$ is dual to the generator $[x_i]$ of $H_1(\partial^{\pm}W; \mathbb{Z})$. The infinite cyclic covering $\widetilde{W}_\varphi$ of $W$ associated with $\varphi$ is obtained from $\widetilde{X}^1$ by attaching infinitely many 0-framed 2-handles along curves homotopic to elements $t^{n}x_{1}b$, $n\in\mathbb{Z}$, in $\pi_1(\widetilde{X}^1)$. Thus, $H_i(\widetilde{W}_\varphi; \mathbb{Q})=H_i(\widetilde{X};\mathbb{Q})$, and they are all finitely generated over $\mathbb{Q}$.
\end{proof}
\smallskip

\begin{Rem}\label{rmk1}
In \cite{CFHH}, the cobordism $W$ constructed in the proof of Theorem \ref{thm1} is used to show that $S_{0}^{3}(K)$ and $S_{0}^{3}(P(K))$ are $\mathbb{Z}$-homology cobordant rel meridians under the same assumption as Theorem \ref{thm1}. So, the cobordism $W$ is a non-trivial cobordism which is simultaneously a $\mathbb{Z}$-homology and $\widetilde{H}$-cobordism. In fact, we can easily find $\widetilde{H}$-cobordisms which are not $\mathbb{Z}$-homology cobordisms. But, we do not know whether every $\mathbb{Z}$-homology cobordism is an $\widetilde{H}$-cobordism.
\end{Rem}

\section{$\mathbb{Z}^\infty$-subgroup in $\mathrm{ker}(\omega)$}
In this section, using the properties of the knot concordance invariants reviewed in Section \ref{KCI} in conjunction with Theorem \ref{thm1}, we establish our theorem on the kernel of the zero-surgery homomorphism $\omega$.

\begin{Th}\label{MT}
The kernel of the zero-surgery homomorphism $\omega:\mathcal{C}\to\Omega(S^1\times S^2)$ contains a subgroup isomorphic to $\mathbb{Z}^\infty$.
\end{Th}

\begin{proof}
Let $P$ be the Mazur pattern shown in Figure \ref{PC}. Let $T_{2,3}$ be a $(2,3)$-torus knot. Note that $V_0(T_{2,3})=1$, $V_0(-T_{2,3})=0$, and $\tau(T_{2,3})=1$. Moreover, $\Upsilon_{T_{2,3}}(t)=-t$ for $t\in[0,1]$. Let $K_n = P^n(T_{2,3})\# -T_{2,3}$.\

We first claim that the family $\{K_n|n\in\mathbb{Z}^+\}$ is mapped to $0$ in $\Omega(S^1\times S^2)$ by $\omega$, and that $K_n$ is a not slice knot for each $n$. By Theorem \ref{thm1}, $S_0^3(T_{2,3})\sim S_0^3(P^{n}(T_{2,3}))$. Then
\begin{eqnarray}
 0 & = & [S_0^3(P^{n}(T_{2,3}))\bigcirc -S_0^3(T_{2,3})]\nonumber\\
 & = & [S_0^3(P^{n}(T_{2,3})\#-T_{2,3})]\nonumber\\
 & = & \omega(P^{n}(T_{2,3})\#-T_{2,3}).\nonumber
\end{eqnarray}
Also, $\tau(K_n)=\tau(P^n(T_{2,3}) \# -T_{2,3})=\tau(P^n(T_{2,3}))-\tau(T_{2,3})=n+1-1=n$ by Theorem \ref{L1}. This proves the claim, showing that the family $\{K_n|n\in\mathbb{Z}^+\}$ is not-trivial elements in $\mathrm{ker}(\omega)$.\

Next, we will show that there is a subset of $\{K_n | n \in \mathbb{Z}^+\}$ which generates a $\mathbb{Z}^\infty$-subgroup in $\mathrm{ker}(\omega)$. By Theorem \ref{BCG}, $$V_0(K_n)=V_0(P^n(T_{2,3})\# -T_{2,3})\le V_0(P^n(T_{2,3}))+V_0(-T_{2,3}).$$
By Remark \ref{rmk1}, $S_{0}^{3}(P^n(T_{2,3}))$ and $S_{0}^{3}(T_{2,3})$ are $\mathbb{Z}$-homology cobordant. Note that $V_0$ is an invariant of the $\mathbb{Z}$-homology cobordism class of the zero-surgery, see \cite[Theorem 3.1]{HKMP}, i.e., if two knots have $\mathbb{Z}$-homology cobordant 0-surgeries, then they have the same $V_0$.  So, $V_0(P^n(T_{2,3})) = V_0(T_{2,3}) = 1$, and hence $V_0(K_n)\le 1$. By Theorem \ref{C1}, $$-2V_0(K_n)\le\Upsilon_{K_n}(t)=\Upsilon_{P^n(T_{2,3})}(t)-\Upsilon_{T_{2,3}}(t).$$ On $[0,\delta_n)$, with sufficiently small $\delta_n$ having no singularity of $\Upsilon_{P^n(T_{2,3})}(t)$, $$\Upsilon_{P^n(T_{2,3})}(t)-\Upsilon_{T_{2,3}}(t)=-(n+1)t+t=-nt,$$ since $\Upsilon'(0)=-\tau$ by Theorem \ref{OSS}(3). This implies $-2V_0(K_n)\le-nt$, and hence $V_0(K_n)>0$. So, $0<V_0(K_n)\le1$. Then $\lim_{n\to\infty}\frac{\tau(K_n)}{V_0(K_n)}=\infty$ because $\tau(K_n)=n$. By Proposition \ref{C2}, there exists a subset of $\{K_n|n\in\mathbb{Z}^+\}$ which generates a $\mathbb{Z}^\infty$-subgroup in $\mathrm{ker}(\omega)$.
\end{proof}

\end{document}